 \documentclass[final,3p,times]{elsarticle}

 \usepackage{epsfig}
 \usepackage{graphicx,color}

\usepackage{amssymb}
\usepackage{amsmath}
 \usepackage{amsthm}
 \usepackage[latin1]{inputenc}

 \newtheorem{theorem}{Theorem}[section]
 \newtheorem{proposition}[theorem]{Proposition}
\newtheorem{lemma}[theorem]{Lemma}
\newtheorem{remark}[theorem]{Remark}

\begin{document}

\begin{frontmatter}




\title{Positive ground state solutions for the critical Klein-Gordon-Maxwell system with potentials}


 \author{Paulo C. Carri\~ao }
 \ead{carrion@mat.ufmg.br}
 \address{Departamento Matem\'atica, UFMG, Belo Horizonte-MG, Brasil}

\author{Patr\'{i}cia L. Cunha\fnref{Cun} }
 \ead{patcunha80@gmail.com}
 \fntext[Cun]{Supported by CAPES/Brazil}
 \address{Departamento Matem\'atica, UFSCar, S\~{a}o Carlos-SP, Brasil}

\author{Ol\'{i}mpio H. Miyagaki\fnref{Mi} }
 \ead{ohmiyagaki@gmail.com}
 \fntext[Mi]{Supported in part by CNPq/Brazil and INCTMat-CNPq/Brazil}
 \address{Departamento Matem\'atica, UFJF, Juiz de Fora-MG, Brasil}

\begin{abstract}
This paper deals with the Klein-Gordon-Maxwell system when the nonlinearity exhibits critical growth. We prove the existence of positive ground
state solutions for this system when a periodic potential $V$ is introduced. The method combines the minimization of the corresponding Euler-Lagrange
functional on the Nehari manifold with the Brézis and Nirenberg technique.
\end{abstract}

\begin{keyword}
Variational methods \sep ground state solutions \sep critical growth.

\MSC 35J47 \sep 35J50 \sep 35B33


\end{keyword}

\end{frontmatter}



\section{Introduction}

In this paper we consider the Klein-Gordon-Maxwell system
\begin{equation*}\label{kgm} \left \{ \begin{array}{ll}
 -\Delta u+ V(x)u-(2\omega +\phi)\phi u=\mu |u|^{q-2}u+|u|^{2^*-2}u &  \mbox{in}\quad \mathbb{R}^{3}\\
\hspace{0.24cm} \Delta \phi=(\omega+\phi)u^{2} & \mbox{in}\quad \mathbb{R}^{3}
\end{array}\right. \tag{$\mathcal{KGM}$}
\end{equation*}
where $\mu$ and $\omega$ are positive real constants, $2<q<2^*=6$ and $u,\phi:\mathbb{R}^{3}\rightarrow \mathbb{R}$. Moreover we assume the
following hypothesis on the continuous function $V$:
\begin{itemize}
\item[(V1)] $ V(x+p)=V(x), \quad x\in\mathbb{R}^3,\, p\in\mathbb{Z}^3$
\item[(V2)]\label{V2}
There exists $V_0>0$ such that $V(x)\geq V_0>0$, $ x\in\mathbb{R}^3$, \newline where $V_0>\frac{2(4-q)}{q-2}\omega^2$ if $2<q<4$.
\end{itemize}

This system appears as a model which describes the nonlinear Klein-Gordon field interacting with the electromagnetic field in the electrostatic case.
The unknowns of the system are the field $u$ associated to the particle and the electric
potential $\phi$. The presence of the nonlinear term simulates the interaction between many particles or external nonlinear perturbations.

Let us recall some previous results that led us to the present research.

The first result is due to Benci and Fortunato. In \cite{Benci-Fortunato-2002}, they proved the existence of infinitely many radially symmetric
solutions for the Klein-Gordon-Maxwell system
\begin{equation}\label{problem 1} \left \{ \begin{array}{ll}
 -\Delta u+ [m_0^2-(\omega +\phi)^2] u=|u|^{q-2}u &  \mbox{in}\quad \mathbb{R}^{3}\\
\hspace{0.24cm} \Delta \phi=(\omega+\phi)u^{2} & \mbox{in}\quad \mathbb{R}^{3}
\end{array}\right.
\end{equation}
considering subcritical behavior on the nonlinearity under the assumptions $|m_0|>|\omega|$ and $4<q<6$. In \cite{D'Aprile-Mugnai}, D'Aprile and Mugnai
covered the case $2 <q< 4$ assuming $m_0\sqrt{p-2}>\sqrt{2}\omega>0$ and the case $q=4$ assuming $m_0>\omega>0$.

Motivated by the approach of Benci and Fortunato, Cassani \cite{Cassani} considered system (\ref{problem 1}) for the critical case by adding a lower
order perturbation:
\begin{equation}\label{problem 2} \left \{ \begin{array}{ll}
 -\Delta u+ [m_0^2-(\omega +\phi)^2] u=\mu|u|^{q-2}u+|u|^{2^*-2}u &  \mbox{in}\quad \mathbb{R}^{3}\\
\hspace{0.24cm} \Delta \phi=(\omega+\phi)u^{2} & \mbox{in}\quad \mathbb{R}^{3}
\end{array}\right.
\end{equation}
where $\mu>0$. He was able to show that
\begin{itemize}
 \item [i)] if $|m_{0}|>|\omega|$ and $4<q<2^{*}$, then for each $\mu>0$, there exists a radially symmetric
 solution for system (\ref{problem 2});
 \item [ii)] if $|m_{0}|>|\omega|$ and $q=4$, then system (\ref{problem 2}) has a radially
 symmetric solution provided that $\mu $ is sufficiently large.
\end{itemize}

The class of (\ref{kgm}) system presented in this paper with such potential $V(x)$ is closely related to a number of several other works. In fact,
the potential $V(x)$ also satisfies the constant case $m_0^2-\omega^2$ which has been extensively considered, see e.g.
\cite{Azzollini-Pisani-Pomponio,Azzollini-Pomponio-KGM,Benci-Fortunato-2002,Cassani,D'Aprile-Mugnai,D'Aprile-Mugnai-nonexistence}.

In \cite{Georgiev-Visciglia}, Georgiev and Visciglia also introduced a class of (\ref{kgm}) system with potentials, however they considered a small
external Coulomb potential in the corresponding Lagrangian density.

We observe that without loss of generality we may assume $\omega>0$, because if $(u,\phi)$ is a solution of the (\ref{kgm}) system, then $(u,-\phi)$
will be a solution corresponding to $-\omega$. Therefore, the sign of $\omega$ is not essential when looking for existence of solutions.

The investigation of \textit{ground state} solutions, that is, couples $(u,\phi)$ which solve (\ref{kgm}) and minimize the action functional
associated to (\ref{kgm}) among all possible nontrivial solutions, has been considered by many authors in a plethora of problems. See, for example,
\cite{Azzollini-Pomponio-KGM,Azzollini-Pomponio-SM,Berestycki-Lions,Li-Wang-Zeng,Zhao-Zhao}.

The authors Azzollini and Pomponio \cite{Azzollini-Pomponio-KGM} established the existence of ground state solutions for the subcritical
Klein-Gordon-Maxwell system (\ref{problem 1}), under the following assumptions:
\begin{itemize}
 \item [i)] $4\leq q<6$ and $m_0>\omega$;
 \item [ii)] $2<q<4$ and $m_0\sqrt{q-2}>\omega\sqrt{6-q}$.
\end{itemize}
Their technique consisted in minimizing the corresponding functional of (\ref{problem 1}) on the Nehari manifold. 

In the present paper we go one step further and extend Theorem 1.1 in \cite{Azzollini-Pomponio-KGM} for the critical growth case. Moreover, we  
establish the sign of the solution.

Our main result is as follows:

\begin{theorem}\label{teorema}
If conditions (V1) and (V2) hold, then the (\ref{kgm}) system has a positive ground state solution for each $\mu>0$ if $4<q<6$ and for
$\mu $ sufficiently large if $2<q\leq 4$.
\end{theorem}

Our approach combines the minimization of the corresponding functional of (\ref{kgm}) system on the Nehari manifold with the Brézis and Nirenberg
technique.


\section{Variational setting}

In this section we introduce notations and prove some preliminary results concerning the variational structure for the (\ref{kgm}) system.

Throughout this paper, $C$ and $C_i$ are positive constants which may change from line to line.

Let us consider the Sobolev space $E$ endowed with the norm
\begin{eqnarray*} \|u\|^{2}=\int_{\mathbb{R}^{3}}(|\nabla u|^2+V(x)u^2)\,dx
\end{eqnarray*}
which is equivalent to the usual Sobolev norm on
$H^1(\mathbb{R}^3)$. Also $\mathcal{D}^{1,2}\equiv
\mathcal{D}^{1,2}(\mathbb{R}^{3})$ represents the completion of
$\mathcal{C}_{0}^{\infty}(\mathbb{R}^{3})$ with respect to the
norm
\begin{eqnarray*} \|u\|_{\mathcal{D}^{1,2}}^{2}=\int_{\mathbb{R}^{3}}|\nabla u|^2\,dx.
\end{eqnarray*}

For any $1\leq s<\infty$, $L^s(\mathbb{R}^3)$ is the usual Lebesgue space endowed with the norm
\begin{eqnarray*}\|u\|_s^s= \int_{\mathbb{R}^3}|u|^s \,dx.
\end{eqnarray*}

Due to the variational nature of the (\ref{kgm}) system, its weak solutions $(u,\phi)\in E\times\mathcal{D}^{1,2}$ are critical points of the
functional $F:E\times \mathcal{D}^{1,2} \rightarrow\mathbb{R}$ defined as
\begin{eqnarray}\label{F} && F(u,\phi) =\frac{1}{2}\int_{\mathbb{R}^{3}}
\Big(|\nabla u|^2-|\nabla \phi|^2+[V(x)-(2\omega+\phi)\phi ]u^2\Big)\,dx-\frac{\mu}{q}\int_{\mathbb{R}^{3}}{|u|^q}\,
 dx-\frac{1}{6}\int_{\mathbb{R}^{3}}{|u|^{6}} \,dx,
\end{eqnarray}
By standard arguments the function $F$ is $C^{1}$ on $E\times
\mathcal{D}^{1,2}$.

In order to avoid the difficulty originated by the strongly indefiniteness of the functional $F$ we apply a \textit{reduction method}, as it has been
done by the aforementioned authors.

\begin{proposition}\label{Propriedade-phi}
For every $u \in E$, there exists a unique $\phi=\phi_u\in \mathcal{D}^{1,2}$ which solves $\Delta \phi=(\omega+\phi)u^{2}$.
Furthermore, in the set $\{x:\,u(x)\neq 0\}$ we have $-\omega\leq\phi_u\leq 0$ for $\omega>0$.
\end{proposition}

\begin{proof} The existence and uniqueness follows from the Lax-Milgram theorem. Using the ideas of \cite{D'Aprile-Mugnai}, fix $u\in E$ and 
consider $\omega >0$. If we multiply both members of $\Delta \phi_u=(\omega+\phi_u)u^{2}$
 by $(\omega+\phi_u)^{-}=\text{min}\{\omega+\phi_u,0\}$, which is an admissible test function, we get
\begin{equation*}
-\int\limits_{\{x|\omega+\phi_u<0\}}{|\nabla\phi_u|^{2}}-\int\limits_{\{x|\omega+\phi_u<0\}}
{(\omega+\phi_u)^{2} u^{2}}= 0
\end{equation*}
so that $\phi_u\geq -\omega$ where $u\neq 0$. 

Finally, by the  Stampacchia's lemma, observe that if $\omega>0$, then $\phi\leq 0$ (for details see \cite{Cassani}).
\end{proof}

According to Proposition \ref{Propriedade-phi}, we can define
\begin{eqnarray*} \Phi :E\rightarrow \mathcal{D}^{1,2}
\end{eqnarray*}
which is of class $C^1$ and maps each $u\in E$ in the unique solution of $\Delta \phi=(\omega+\phi)u^{2}$. 

From the definition of 
$\Phi$ we have
\begin{eqnarray}\label{F'_phi}
F'_{\phi}(u,\phi_u)=0, \quad\forall u\in E.
\end{eqnarray}

Now let us consider the functional
\begin{eqnarray}\label{I(u)=F(u,phi)}
 I: E\rightarrow \mathbb{R}, \quad I(u) := F(u,\phi_u),
 \end{eqnarray}
then $I\in C^1(E,\mathbb{R})$ and, by (\ref{F'_phi}), 
\begin{eqnarray*}I'(u)=F_u'(u,\phi_u).
\end{eqnarray*}

Multiplying both members of the second equation in the ($\mathcal{KGM}$) system by $\phi_u$ and integrating by parts, we obtain
\begin{eqnarray}\label{system2-modificado2}
\int_{\mathbb{R}^{3}} {| \nabla\phi_u|^{2}}\,dx = -\int_{\mathbb{R}^{3}}{\omega\phi_u} u^{2}\,dx-\int_{\mathbb{R}^{3}}{\phi_u^{2}u^{2}}\,dx.
\end{eqnarray}

\begin{remark}\label{remark-phi-u} Let us note that
\begin{eqnarray*}
\|\phi_u\|_{\mathcal{D}^{1,2}}^2 \leq \omega\int_{\mathbb{R}^3}|\phi_u| u^2 \,dx \leq \omega\|\phi_u\|_6 \|u\|_{12/5}^2
\end{eqnarray*}
then
\begin{eqnarray*}\|\phi_u\|_{\mathcal{D}^{1,2}}\leq C\omega
\|u\|_{12/5}^2.
\end{eqnarray*}
\end{remark}

By the definition of $F$ and using (\ref{system2-modificado2}), the functional $I$ may be written as
\begin{eqnarray} \label{I}
 I(u)=\frac{1}{2}\int_{\mathbb{R}^{3}}\Big({|\nabla u|^{2}+ V(x) u^{2}} -\omega\phi_u u^{2}\Big)\,dx
-\frac{\mu}{q}\int_{\mathbb{R}^{3}}|u|^{q}\,dx -\frac{1}{6}\int_{\mathbb{R}^{3}}{u^{6}}\,dx
\end{eqnarray}
while for $I'$ we have,
\begin{eqnarray}\label{I'}
\langle I'(u),v\rangle = \int_{\mathbb{R}^{3}}\Big(\nabla u\cdot
\nabla v+ V(x)uv -(2\omega+\phi_u)\phi_u uv -\mu {|u|^{q-2}}uv-
u^{5}v \Big)\,dx.
\end{eqnarray}
for every $u,v\in E$. Then, $(u,\phi)\in E\times \mathcal{D}^{1,2}$ is a weak solution of $(\mathcal{KGM})$ if, and only if, $\phi=\phi_u$ and
$u\in E$ is a critical point of $I$. The functional $I$ obtained is not strongly indefinite anymore and we will look for its critical points.


\section{Proof of Theorem \ref{teorema}}

The purpose of this section is to obtain critical points of the functional $I$, then we shall consider the correspondent Nehari manifold
\begin{eqnarray}\label{Nehari}\mathcal{N}=\{u\in E\setminus \{0\}\hspace{.1cm}|\hspace{.1cm} G(u)=0\},
\end{eqnarray}
where
\begin{eqnarray*}G(u)&=&\langle I'(u), u\rangle \\
&=&\int_{\mathbb{R}^{3}}( |\nabla u|^2+V(x) u^2)\,dx -\int_{\mathbb{R}^{3}}(2\omega+\phi_u)\phi_u u^2\,dx
-\mu\int_{\mathbb{R}^{3}} |u|^q\,dx -\int_{\mathbb{R}^{3}} u^{6}\,dx.
\end{eqnarray*}

The next lemma will be useful when proving that $\mathcal{N}$ is a Nehari manifold of $C^1$ class:

\begin{lemma}\label{2.2a} Let $u\in E$ and $2\psi_u=\Phi'(u)\in \mathcal{D}^{1,2}(\mathbb{R}^3)$.
Then, $\psi_u$ is a solution of the integral equation
\begin{eqnarray*}\int_{\mathbb{R}^3}\omega\psi_uu^2\,dx=\int_{\mathbb{R}^3}(\omega+\phi_u)\phi_u u^2\,dx
\end{eqnarray*}
and, as a consequence, $\psi_u\leq 0$.
\end{lemma}

\begin{proof} The proof follows from the fact that $\psi_u$ satisfies
$$\Delta\psi_u-u^2\psi_u=(\omega+\phi_u)u^2,
$$
as we know by \cite{D'Aprile-Mugnai-nonexistence}.

\end{proof}

Now we need some results concerning the Nehari manifold.

\begin{lemma}\label{norma(u)>=C} There exists a constant $C>0$ such that $\|u\|\geq C$, for all $u\in \mathcal{N}$.
\end{lemma}

\begin{proof} Let $u\in \mathcal{N}$, then using the H\"{o}lder inequality
\begin{eqnarray*}
0&=&\|u\|^{2}-2\int_{\mathbb{R}^{3}} \omega\phi_u u^2\,dx -
\int_{\mathbb{R}^{3}} \phi_u^2 u^2\,dx-\mu\|u\|_{q}^{q}-\|u\|_{6}^{6}\\
 & \geq & \|u\|^{2}-\mu C_1 \|u\|^{q}-C_2\|u\|^{6}
\end{eqnarray*}
\noindent and so, there exists $C>0$ such that $\|u\| \geq C$.
\end{proof}

\begin{lemma}\label{nNehari} $\mathcal{N}$ is a $C^1$ manifold.
\end{lemma}

\begin{proof} Consider
\begin{eqnarray*}2I(u)=\|u\|^2-\int_{\mathbb{R}^{3}}\omega\phi_u u^2\,dx
-\frac{2\mu}{q}\int_{\mathbb{R}^{3}} |u|^q\,dx-\frac{1}{3}\int_{\mathbb{R}^{3}} {u}^{6}\,dx
\end{eqnarray*}
then, for all $u\in E$,
\begin{eqnarray*}G(u)=2I(u)-\int_{\mathbb{R}^{3}}\omega\phi_u u^2\,dx-\int_{\mathbb{R}^{3}}\phi_u^2 u^2 \,dx+
\frac{(2-q)\mu}{q}\int_{\mathbb{R}^{3}}|u|^q\,dx-\frac{2}{3}\int_{\mathbb{R}^{3}} u^6\,dx.
\end{eqnarray*}

Let us prove that there exists $C>0$ such that $\langle G'(u),u\rangle \leq -C$, for all $u\in \mathcal{N}$.

$G$ turns out to be a $C^1$ functional then, using Lemma \ref{2.2a},
\begin{eqnarray*}
\langle G'(u),u\rangle &=& \langle 2I'(u),u\rangle+(2-q)\mu\int_{\mathbb{R}^{3}}
|u|^q\,dx-4\int_{\mathbb{R}^{3}} u^{6}\,dx-4\int_{\mathbb{R}^{3}}(\omega+\phi_u+\psi_u)\phi_u u^2\,dx\\
&=& (2-q)\|u\|^2-(2-q)\int_{\mathbb{R}^{3}}(2\omega+\phi_u)\phi_uu^2\,dx-(2-q)\int_{\mathbb{R}^{3}} u^{6}\,dx+\\
&&-4\int_{\mathbb{R}^{3}} u^{6}\,dx -4\int_{\mathbb{R}^{3}}(\omega+\phi_u+\psi_u)\phi_u u^2\,dx\\
&\leq & (2-q)\|u\|^2-\int_{\mathbb{R}^{3}}[(2-q)(2\omega+\phi_u)+4(\omega+\phi_u+\psi_u)]\phi_uu^2\,dx\\
&=&  (2-q)\|u\|^2-\int_{\mathbb{R}^{3}}[2(4-q)\omega+(6-q)\phi_u+4\psi_u]\phi_uu^2\,dx
\end{eqnarray*}

The case $4\leq q<6$ is trivial. Consider $2< q<4$. Using Lemma \ref{norma(u)>=C}, condition (V2) and Proposition \ref{Propriedade-phi}, we obtain:
\begin{eqnarray*}
\langle G'(u),u \rangle &\leq &  (2-q)\int_{\mathbb{R}^{3}}|\nabla
u|^2\,dx+(2-q)\int_{\mathbb{R}^{3}}V_0u^2\,dx
-2(4-q)\int_{\mathbb{R}^{3}}\omega\phi_u
u^2\,dx\\
&= & (2-q)\int_{\mathbb{R}^{3}}|\nabla
u|^2\,dx+\int_{\mathbb{R}^{3}}[(2-q)V_0-2(4-q)\omega\phi_u]
u^2\,dx\\
&\leq & (2-q)\int_{\mathbb{R}^{3}}|\nabla
u|^2\,dx+\int_{\mathbb{R}^{3}}[(2-q)V_0+2(4-q)\omega^2]
u^2\,dx\\
&\leq & -C.
\end{eqnarray*}
where $C$ is a positive constant.

\end{proof}

\begin{lemma}\label{I|N >= C} $I$ is bounded from below on $\mathcal{N}$ by a positive constant.
\end{lemma}

\begin{proof} For any $u\in\mathcal{N}$,
\begin{eqnarray}\label{I|N}
I\big |_{\mathcal{N}}(u)=\frac{q-2}{2q}\|u\|^2+\frac{4-q}{2q}\int_{\mathbb{R}^{3}}
\omega\phi_u u^2\,dx+\frac{1}{q}\int_{\mathbb{R}^{3}}\phi_u^2 u^2\,dx+\frac{6-q}{6q}\int_{\mathbb{R}^{3}} u^{6}\,dx.
\end{eqnarray}

We have to distinguish two cases. If $4\leq q < 6$, then each term in (\ref{I|N}) is positive and we get
\begin{eqnarray*}
I\big |_{\mathcal{N}}(u)\geq \frac{q-2}{2q}\|u\|^2.
\end{eqnarray*}

Otherwise, if $2<q<4$, we use Proposition \ref{Propriedade-phi} and condition (V2) to obtain
\begin{eqnarray*}
I\big |_{\mathcal{N}}(u)&\geq &
\frac{q-2}{2q}\int_{\mathbb{R}^{3}}|\nabla u|^2\,dx+
\frac{q-2}{2q}\int_{\mathbb{R}^{3}}
V(x)u^2\,dx-\frac{4-q}{2q}\int_{\mathbb{R}^{3}}\omega^2 u^2\,dx\\
&\geq & \frac{q-2}{2q}\int_{\mathbb{R}^{3}}|\nabla u|^2\,dx+
\frac{1}{2q}
\int_{\mathbb{R}^{3}}[(q-2)V_0-(4-q)\omega^2]u^2\,dx\\
&\geq & C\|u\|^2.
\end{eqnarray*}

The conclusion follows by Lemma \ref{norma(u)>=C}.
\end{proof}

By the Ekeland Variational Principle, there exists a minimizing sequence $(u_n)\subset\mathcal{N}$, which can be considered a $(PS)_c$ sequence, i.e.,
\begin{eqnarray}\label{I(u_n)->c}I(u_n)\rightarrow c \qquad \mbox{and}
\qquad I'(u_n)\rightarrow 0,
\end{eqnarray}
where $c$ is characterized by
\begin{eqnarray}\label{valorC}\displaystyle c:=\inf_{\gamma\in\Gamma}
\max_{0\leq
t\leq 1}I(\gamma(t))
\end{eqnarray}
\noindent and
\begin{eqnarray*}\Gamma=\{\gamma\in\mathcal{C}([0,1],E)|I(\gamma(0))=0,
I(\gamma(1))<0 \}.
\end{eqnarray*}

\begin{lemma}\label{c} The number $c$ given in (\ref{valorC}) satisfies
\begin{eqnarray} 0<c<\frac{1}{3}S^\frac{3}{2},
\end{eqnarray}
where S is the best Sobolev constant, namely
\begin{eqnarray*}
S:=\inf\limits_{ u\in \mathcal{D}^{1,2}(\mathbb{R}^{3})\atop_{u\neq 0}}\frac{\int_{\mathbb{R}^{3}}|\nabla u|^{2}\,dx}
{\Big( \int_{\mathbb{R}^{3}}{u^{6}} \,dx\Big)^\frac{1}{3}}.
\end{eqnarray*}

\end{lemma}

\begin{proof}
This proof uses a technique by Br\'{e}zis and Nirenberg \cite{Brezis-Nirenberg} and some of its variants. For the sake of completeness we give a sketch of
the proof, see \cite{Carriao-Cunha-Miyagaki} and \cite{Olimpio}.

It suffices to show that
\begin{eqnarray}\label{suffices}
\sup_{t\geq 0} I(tv_{0})< \frac{1}{3} S^{\frac{3}{2}}
\end{eqnarray}
for some $v_{0}\in E, v_{0}\neq 0 $.

Indeed, from Proposition \ref{Propriedade-phi} $ii)$, observing that
\begin{eqnarray*}I(tv_{0})\leq  \frac{t^2}{2}\|v_0\|^2+\frac{t^2}{2}\int_{\mathbb{R}^3}\omega^2 v_0^2\,dx-\frac{\mu}{q}t^q
\int_{\mathbb{R}^3}|v_0|^q\,dx-\frac{t^6}{6}\int_{\mathbb{R}^3}v_0^6\,dx
\end{eqnarray*}
we have $\displaystyle\lim_{t\rightarrow +\infty}I(tv_{0})= -\infty$. Hence,
\begin{eqnarray}\label{sup} c=\displaystyle \inf_{\gamma\in\Gamma}
\max_{0\leq
t\leq 1}I(\gamma(t))\leq\sup_{t\geq 0} I(tv_{0})<\frac{1}{3}
S^{\frac{3}{2}}.
\end{eqnarray}

In order to prove (\ref{suffices}) consider $R>0$ fixed and a cut-off function $\varphi\in C_{0}^{\infty}$ such that
\begin{eqnarray*}\varphi|B_{R}=1, \hspace{0.3cm} 0\leq\varphi\leq
1 \hspace{0.1cm}\mbox{in}\hspace{0.1cm} B_{2R}\hspace{0.3cm}
\mbox{and}\hspace{0.3cm} \mbox{supp}\hspace{0.06cm} \varphi\subset
B_{2R}.
\end{eqnarray*}
where $B_R$ is a ball in $\mathbb{R}^3$ centered in zero with radius $R$.

Let $\varepsilon >0$ and define $w_{\varepsilon}:=u_{\varepsilon}\varphi$ where $u_{\varepsilon}(x)=C\varepsilon^{1/4}/
( \varepsilon+|x|^{2} )^{1/2}$ is the well known Talenti's function in dimension $N=3$ (see \cite{Talenti}) and also consider $v_{\varepsilon}\in
C_{0}^{\infty}$ given by
\begin{eqnarray}\label{v-epsilon definicao}
v_{\varepsilon}:=\frac{w_{\varepsilon}}{\|w_{\varepsilon}\|_{L^{6}(B_{2R})}}.
\end{eqnarray}

From the estimates given in \cite{Brezis-Nirenberg} we have, as $\varepsilon\rightarrow 0$,
\begin{eqnarray}\label{X-epsilon}
X_{\varepsilon}:=\|\nabla v_{\varepsilon}\|_{2}^{2}\leq S+O(\varepsilon^{\delta}), \hspace{0.3cm} \mbox{where}\,\, \delta=\frac{1}{2}.
\end{eqnarray}

Since $\lim \limits_{t\rightarrow\infty} I(tv_{\varepsilon})=-\infty \hspace{0.1cm} \forall\varepsilon$, there exists $t_{\varepsilon}\geq 0$ such that
$\sup\limits_{t\geq 0}I(tv_{\varepsilon})=I(t_{\varepsilon}v_{\varepsilon })$ and we may assume without loss of generality that $t_\varepsilon>0$.

\bigskip

\noindent\textbf{Claim 1}. The following estimate holds
\begin{eqnarray}t_{\varepsilon}\leq \Big( \int_{\mathbb{R}^3}|\nabla
v_{\varepsilon}|^{2}\,dx+\int_{B_{2R}}
(V(x)+2\omega^2)v_{\varepsilon}^{2}\,dx \Big)^\frac{1}{4}:=r_{\varepsilon}>0.
\end{eqnarray}

\noindent\textit{Proof of Claim 1}:
Letting $\gamma(t):=I(tv_{\varepsilon})$ and using the Proposition \ref{Propriedade-phi} $ii)$,
\begin{eqnarray*}
\gamma'(t)&=& \langle I'(tv_{\varepsilon}),v_{\varepsilon}\rangle\\
&=& t\int_{\mathbb{R}^3}|\nabla v_\varepsilon|^2\,dx+t\int_{B_{2R}}V(x)v_\varepsilon^2\,dx - t\int_{B_{2R}}2\omega\phi_{tv_{\varepsilon}}
{v_{\varepsilon}}^2\,dx+\\
&&-t\int_{B_{2R}}\phi_{tv_{\varepsilon}}^2 v_{\varepsilon}^{2}\,dx-\mu
t^{q-1}\int_{B_{2R}} |{v_{\varepsilon}|^q}\,dx-t^{5}\\
&\leq & t r_{\varepsilon}^{4}-t^{5}-t\int_{B_{2R}}\phi_{tv_{\varepsilon}}^2 v_{\varepsilon}^{2}\,dx-\mu
t^{q-1}\int_{B_{2R}} |{v_{\varepsilon}|^q}\,dx,
\end{eqnarray*}
which is negative for $t>r_\varepsilon$.

Now, the function of $t$: $\frac{t^{2}}{2}r_{\varepsilon}^{4}-\frac{t^{6}}{6}$ is increasing on $[0,r_{\varepsilon})$, hence using
(\ref{X-epsilon}), H\"{o}lder inequality and Remark \ref{remark-phi-u} we conclude that
\begin{eqnarray*}
I(t_{\varepsilon}v_{\varepsilon})& \leq & \frac{1}{3}\Big(S+O(\varepsilon^{\delta})+\int_{B_{2R}}(V(x)+2\omega^2) v_{\varepsilon}^{2} \,dx \Big)^{3/2} +\\
&&+ Ct_{\varepsilon}^{4}\|v_{\varepsilon}\|^4_\frac{12}{5} - \frac{\mu}{q}t_{\varepsilon}^{q}\int_{B_{2R}}{v_{\varepsilon}^{q}} \,dx.
\end{eqnarray*}

Applying the inequality
\begin{eqnarray*}
(a+b)^{\alpha}\leq a^{\alpha}+\alpha(a+b)^{\alpha-1}b,
\end{eqnarray*}
which is valid for $a,b\geq 0$, $\alpha\geq 1$, we obtain
\begin{eqnarray*}
I(t_{\varepsilon}v_{\varepsilon}) &\leq & \frac{1}{3}S^{\frac{3}{2}}+O(\varepsilon^{\delta}) +C_1\int_{B_{2R}}(V(x)+2\omega^2)v_{\varepsilon}^{2}\,dx+ \\
&& + C_2 C_{\varepsilon}^{\frac{4}{q}}\|v_{\varepsilon}\|^{4}_{\frac{12}{5}}- \mu C_{\varepsilon}\int_{B_{2R}}{v_{\varepsilon}^{q}}\,dx ,
\end{eqnarray*}
where $C_{\varepsilon}=t_{\varepsilon}^q/q\geq C_{0}^q/q>0$.

We contend that

\bigskip

\noindent\textbf{Claim 2}.
\begin{eqnarray}\label{star}
\lim_{\varepsilon\rightarrow 0}\frac{1}{\varepsilon^{\delta}}\Big(C_1 \int_{B_{2R}}((V(x)+2\omega^2) v_{\varepsilon}^{2}-C_2\mu v_{\varepsilon}^{q})\,dx
+C_3\|v_{\varepsilon}\|_{\frac{12}{5}}^{4}  \Big)=-\infty.
\end{eqnarray}

Assuming (\ref{star}) for a while we have
\begin{eqnarray*}
I(t_{\varepsilon}v_{\varepsilon})< \frac{1}{3}S^{\frac{3}{2}}, \hspace{0.3cm} \varepsilon\hspace{0.1cm} \mbox{small}
\end{eqnarray*}
showing (\ref{suffices}) and thus Lemma \ref{c}.

\bigskip

\noindent\textit{Proof of Claim 2}:

As in \cite{Brezis-Nirenberg}, we obtain
\begin{eqnarray*}
\int_{B_{2R}}|w_{\varepsilon}|^{6}\,dx=C \int_{\mathbb{R}^{3}} \frac{1}{(1+|x|^{2})^{3}}\,dx+O(\varepsilon^{\frac{3}{2}})
\end{eqnarray*}
so, in view of (\ref{v-epsilon definicao}), it suffices evaluate (\ref{star}) with $w_{\varepsilon}$ instead of $v_{\varepsilon}$. In order to prove
(\ref{star}) we must show
\begin{eqnarray}\label{star2}
\lim\limits_{\varepsilon\rightarrow 0}\frac{1}{\varepsilon^{\delta}}\Big[ \int_{B_{R}}((V(x)+2\omega^2)w_{\varepsilon}^{2}-\mu
w_{\varepsilon}^{q})\,dx +\Big(\int_{B_{R}}|w_{\varepsilon}|^{\frac{12}{5}}\,dx \Big)^{\frac{5}{3}} \Big]=-\infty
\end{eqnarray}
and also that
\begin{eqnarray}\label{star22}
\frac{1}{\varepsilon^\delta} \Big[\int_{B_{2R} \setminus B_{R}}((V(x)+2\omega^2)v_{\varepsilon}^{2}-\mu
v_{\varepsilon}^{q})\,dx +\Big(\int_{B_{2R}\setminus B_{R}}|v_{\varepsilon}|^{\frac{12}{5}}\,dx \Big)^{\frac{5}{3}}\Big]
\end{eqnarray}
is bounded.

\bigskip

\noindent \textit{Verification of (\ref{star2})}. Let
\begin{eqnarray*}
I_{\varepsilon}:=\frac{1}{\varepsilon^{\delta}}\Big[ \int_{B_{R}}((V(x)+2\omega^2)w_{\varepsilon}^{2}-\mu
w_{\varepsilon}^{q})\,dx +\Big(\int_{B_{R}}|w_{\varepsilon}|^{\frac{12}{5}}\,dx \Big)^{\frac{5}{3}}\Big].
\end{eqnarray*}

At first, using the fact that $V(x)$ is continuous, and hence, $V\in L_{loc}^{\infty}$, we get
\begin{eqnarray*}
I_{\varepsilon}\leq \frac{1}{\varepsilon^{\delta}}\Big[C\|V\|_{L^{\infty}(B_{R})} \int_{B_{R}}(w_{\varepsilon}^{2}-\mu w_{\varepsilon}^{q})\,dx
+\Big(\int_{B_{R}}|w_{\varepsilon}|^{\frac{12}{5}}\,dx \Big)^{\frac{5}{3}}\Big].
\end{eqnarray*}

Now, on $B_{R}$, by changing variables we have
\begin{eqnarray}\label{aqui}
I_{\varepsilon}&\leq & \varepsilon^{1-\delta}\Big[ C_{1} \int_{0}^{\frac{R}{\sqrt{\varepsilon}}}\frac{r^{2}}{1+r^{2}} \,dr -
\mu C_{2}\varepsilon^{\frac{2-q}{4}} \int_{0}^{\frac{R}{\sqrt{\varepsilon}}}\frac{r^{2}}{(1+r^{2})^{\frac{q}{2}} } \,dr \nonumber\\
&+& C_{3}\varepsilon^{\frac{1}{2}}\Big( \int_{0}^{\frac{R}{\sqrt{\varepsilon}}}\frac{r^{2}}{(1+r^{2})^{\frac{6}{5}} } \,dr \Big)^{\frac{5}{3}} \Big],
\end{eqnarray}
where $C_{i}$ are positive constants independent from $\varepsilon$.

By simple computations, one gets
\begin{eqnarray*}
\int_{0}^{\frac{R}{\sqrt{\varepsilon}}} \frac{r^{2}}{(1+r^{2})^{\frac{6}{5}}} \,dr
\leq \int_{0}^{\frac{R}{\sqrt{\varepsilon}}}\frac{r^{2}}{1+r^{2}}\, dr=
\frac{R}{\sqrt{\varepsilon}}-\arctan(\frac{R}{\sqrt{\varepsilon}})
\end{eqnarray*}
then,
\begin{eqnarray*}
I_{\varepsilon}&\leq & C_{1}R-C_{1}\varepsilon^{\frac{1}{2}}\arctan(\frac{R}{\sqrt{\varepsilon}})
-\mu C_{2}\varepsilon^{\frac{4-q}{4}}\int_{0}^{\frac{R}{\sqrt{\varepsilon}}} \frac{r^{2}}{(1+r^{2})^{\frac{q}{2}}} dr +\\
&&+C_{3}\varepsilon\Big(\int_{0}^{\frac{R}{\sqrt{\varepsilon}}} \frac{r^{2}}{(1+r^{2})^{\frac{6}{5}}} dr  \Big)^{\frac{5}{3}}\\
&\leq & C_{1}R-\mu C_{2}\varepsilon^{\frac{4-q}{4}}\int_{0}^{\frac{R}{\sqrt{\varepsilon}}} \frac{r^{2}}{(1+r^{2})^{\frac{q}{2}}}
dr+C_{3}R^{\frac{5}{3}}\varepsilon^{\frac{1}{6}}.
\end{eqnarray*}

We have to distinguish two cases: either $2<q\leq 4$ or $4<q<6$.

The case $4<q<6$ was proved by Cassani \cite{Cassani}. However, we can also show (\ref{star2}) using the last inequality, since the integral
$\textstyle\int_{0}^{\infty}\frac{r^{2}}{(1+r^{2})^{q/2}} \,dr$ converges.

If $2<q\leq 4$ and noting that $\textstyle\int_{0}^{\infty}\frac{r^{2}}{(1+r^{2})^{q/2}} dr\geq \frac{\pi}{4}$ we conclude
\begin{eqnarray*}
I_{\varepsilon}\leq C_{4}-\frac{\pi}{4}\mu C_{2}\varepsilon^{\frac{4-q}{4}}.
\end{eqnarray*}

Finally, choosing $\mu=\varepsilon^{-\frac{1}{2}}$, we infer that $I_{\varepsilon}\rightarrow -\infty$ as $\varepsilon\rightarrow 0$. Hence this proves
(\ref{star2}).

\bigskip

\noindent \textit{Verification of (\ref{star22})}. We have
\begin{eqnarray*}
&&\hspace{-1cm}\frac{1}{\varepsilon^\delta} \Big[\int_{B_{2R}\setminus B_{R}}((V(x)+2\omega^2)v_{\varepsilon}^{2}\,dx
-\mu v_{\varepsilon}^{q})\,dx +\Big(\int_{B_{2R}\setminus B_{R}}v_{\varepsilon}^{12/5}\,dx \Big)^{\frac{5}{3}}\Big]  \\
&\leq & \frac{C_{1}}{\varepsilon^{\delta}}\int_{B_{2R}\setminus B_{R}}\varphi^{2}u_{\varepsilon}^{2}\,dx+
\frac{C_{2}}{\varepsilon^{\delta}}\Big(\int_{B_{2R}\setminus B_{R}}\varphi^{12/5}u_{\varepsilon}^{12/5} \,dx\Big)^{\frac{5}{3}}\\
&\leq & C_{1}\varepsilon \|\varphi\|^{2}_{H^1(B_{2R}\setminus B_{R})} + C_{2}\varepsilon^{2+\delta}\|\varphi^{6/5} \|^{5/3}_{H^1(B_{2R}\setminus B_{R})}
\end{eqnarray*}
where we choose R large such that $u_{\varepsilon}^{2}\leq \varepsilon^{1+\delta}$, $\forall \,|x|\geq R$. Then we conclude that equation (\ref{star22}) is
bounded.

Consequently, the proof of Claim 2 is complete.
\end{proof}

Now we show that the functional $I$ satisfies the structural assumptions of the Mountain Pass Theorem as well as the behavior of the $(PS)$ sequence.

\begin{lemma}[Mountain Pass Geometry]\label{Mountain Pass} The functional $I$ satisfies the following conditions:
\begin{itemize}
 \item [(i)] There exist positive constants $\beta,\rho$ such that
 $I(u)\geq\beta$ for $\|u\|=\rho$.
 \item [(ii)] There exists $u_{1}\in E$ with $\|u_{1}\|>\rho$ such that
 $I(u_{1})< 0$.
\end{itemize}
\end{lemma}

\begin{proof} The proof of this lemma can be found in \cite{Carriao-Cunha-Miyagaki}, but we exhibit it here for completeness.

Using the Sobolev embeddings, we have
\begin{eqnarray*}
I(u)\geq C_{1}\|u\|^2-C_{2}\|u\|^q-C_{3}\|u\|^{6},
\end{eqnarray*}
where $C_{1}$, $C_{2}$ and $C_{3}$ are positive constants. Since $q>2$, there exists $\beta,\rho >0$ such that $\inf\limits_{\|u\|=\rho}I(u)>\beta$,
showing $(i)$.

Let $u\in E$, then for $t\geq 0$ and from Proposition (\ref{Propriedade-phi}) we conclude
\begin{eqnarray*}
I(tu)\leq C_{4} t^2\|u\|^2 +\frac{\omega^2}{2}t^2\|u\|_{2}^2-\frac{\mu}{q}t^q\|u\|_{q}^q- \frac{1}{6}t^{6}\|u\|_{6}^{6}.
\end{eqnarray*}

Since $q>2$, there exists $u_{1}\in E$, $u_{1}:=tu$ with $t$ sufficiently large such that $\|u_{1}\|>\rho$ and $I(u_{1})< 0$, proving (ii).
\end{proof}

Now, by using the Ambrosetti-Rabinowitz Mountain Pass Theorem \cite{Ambrosetti-Rabinowitz}, there exists a $(PS)_c$ sequence $(u_n)$ as in
(\ref{I(u_n)->c}).

\begin{lemma}\label{Lemma PSc bounded}
The $(PS)_{c}$ sequence $(u_{n})$ is bounded in $E$.
\end{lemma}

\begin{proof}
By hypothesis, let $(u_{n})\subset E$ be such that $-\langle I'(u_n),u_n\rangle \leq o_n(1)\|u_{n} \|$ and $|I(u_{n})|\leq M$, for some positive constant
$M$.
Then from (\ref{I}) and (\ref{I'}),
\begin{eqnarray}\label{ps-bounded-inequality}
&&\hspace{-0.5cm} qM + o_n(1)\|u_{n} \| \geq  q I(u_{n})-\langle I'(u_{n}),u_{n}\rangle = \nonumber \\
&&= \Big(\frac{q}{2}-1\Big)\int_{\mathbb{R}^{3}}{\Big( |\nabla u_{n}|^{2}+ V(x)u_{n}^{2}\Big)\,dx } +\Big(2-\frac{q}{2}\Big)
\int_{\mathbb{R}^{3}}{\omega \phi_{u_n}u_{n}^{2}}\,dx+\nonumber\\
&&\hspace{0.5cm}+ \int_{\mathbb{R}^{3}}{\phi_{u_n}^{2}u_{n}^{2}}\,dx
+\Big(1-\frac{q}{6}\Big)\int_{\mathbb{R}^{3}}{u_{n}^{6}}\,dx \nonumber\\
&& \geq \Big(\frac{q-2}{2}\Big)\int_{\mathbb{R}^{3}}{\Big( |\nabla u_{n}|^{2}+ V(x)u_{n}^{2}\Big) }\,dx
-\omega\Big(\frac{q-4}{2}\Big) \int_{\mathbb{R}^{3}}{ \phi_{u_n}u_{n}^{2}}\,dx.
\end{eqnarray}

As in Lemma \ref{I|N >= C}, there are two cases to be considered: either $2<q<4$ or $4\leq q < 6$.

If $4\leq q < 6$, then by Proposition \ref{Propriedade-phi} and inequality (\ref{ps-bounded-inequality})
\begin{eqnarray*}
qM + o_n(1)\|u_{n} \| &\geq & C\|u_{n} \|^{2}+ \omega\Big(\frac{q-4}{2}\Big)\int_{\mathbb{R}^{3}}{(-\phi_{u_n}) u_{n}^{2} }\,dx\\
&\geq& C\|u_{n} \|^{2}
\end{eqnarray*}
and we deduce that $(u_{n} )$ is bounded in $E$.

But if $2<q<4$, then from (\ref{ps-bounded-inequality}), Proposition \ref{Propriedade-phi} and condition (V2) we get
\begin{eqnarray*}
qM + o_n(1)\|u_{n} \| &\geq &
\Big(\frac{q-2}{2}\Big)\int_{\mathbb{R}^{3}}{ |\nabla u_{n}|^{2}}\,dx+ \Big(\frac{(q-2)V_{0}+(q-4)\omega^{2}}{2}
\Big)\int_{\mathbb{R}^{3}}{u_{n}^{2}}\,dx\\
&\geq& C\|u_{n} \|^{2},
\end{eqnarray*}
which again implies that $(u_{n} )$ is bounded in $E$.
\end{proof}

\begin{lemma}\label{lema de compacidade}
There exist $C>0$, $r>0$ and $\xi\in\mathbb{R}^3$ such that
\begin{eqnarray*}\int_{B_{r}(\xi)}u_n^2 \,dx\geq C,
\end{eqnarray*}
where $(u_n)\subset\mathcal{N}$ is a minimizing sequence.
\end{lemma}

\begin{proof} Let $(u_n)$ be a minimizing sequence in $\mathcal{N}$. Suppose by contradiction that there exists $ \bar{r}>0 $ such that
\begin{eqnarray*}
\lim_n\sup_{}\int_{B_{\bar{r}}(\xi)}u_n^2\,dx=0.
\end{eqnarray*}

Using Lemma I.1 of  \cite{Lions1} and the previous lemma, it follows that, for $2<q<6$,
\begin{eqnarray*}
\int_{\mathbb{R}^3}{|u_n|^q} \,dx \rightarrow 0, \quad n\rightarrow\infty.
\end{eqnarray*}

Next we claim that
\begin{eqnarray}\label{claim}
\|u_n\|^2=\int_{\mathbb{R}^3} u_n^{6}\,dx + o_n(1).
\end{eqnarray}

Indeed, noting that
\begin{eqnarray*}
\langle I'(u_n),u_n\rangle =
\| u_n\|^2-\int_{\mathbb{R}^{3}}(2\omega+\phi_{u_n})\phi_{u_n} u_n^2\,dx -\mu\int_{\mathbb{R}^{3}} {|u_n|^q}\,dx -\int_{\mathbb{R}^{3}} u_n^{6}\,dx
\end{eqnarray*}
and from (\ref{system2-modificado2}), we infer
\begin{eqnarray*}
-\int_{\mathbb{R}^{3}}(2\omega+\phi_{u_n})\phi_{u_n} u_n^2\,dx
&\leq &  -2\int_{\mathbb{R}^{3}}\omega\phi_{u_n} u_n^2 \,dx
\leq C\|\phi_{u_n}\|_{\mathcal{D}^{1,2}} \|{u_n}\|_{ \frac{12}{5}}^2\\
& \leq & C\|u_n\|_{ \frac{12}{5}}^2 ,
\end{eqnarray*}
which converges to zero as $n\rightarrow \infty$. Then (\ref{claim}) holds.

Assume $\|u_n\|^2 \rightarrow \ell >0$, as $n\rightarrow\infty$. Since $I(u_n)\rightarrow c $,
\begin{eqnarray*}
\frac{1}{2}\|u_n\|^2-\frac{1}{6}\int_{\mathbb{R}^{3}} u_n^{6}\,dx\rightarrow c
\end{eqnarray*}
hence $c=\frac{1}{3}\ell$.

On the other hand, by the definition of $S$, we have
\begin{eqnarray*}
\ell\geq S\ell^{1/3}  \Rightarrow \ell\geq S^{3/2}.
\end{eqnarray*}

But since $c=\frac{1}{3}\ell\geq \frac{1}{3}S^{3/2}$, we have a contradiction. Therefore, $\|u_n\|^2\rightarrow 0$, which is in contradiction with
Lemma \ref{norma(u)>=C}, then $(u_n)$ does not vanish and Lemma \ref{lema de compacidade} holds.
\end{proof}

\begin{lemma}\label{azz+pomp}
If $u_n\rightharpoonup u_0$ weakly in $E$ then, up to subsequences, $\phi_{u_n}\rightharpoonup \phi_{u_0}$ weakly in $\mathcal{D}^{1,2}$.  As a consequence
$I'(u_n)\rightarrow I'(u_0)$, as $n\rightarrow\infty$.
\end{lemma}

\begin{proof}
The proof is an easy adaptation of \cite{Azzollini-Pomponio-KGM}, but for the sake of completeness we give a sketch of it.

Let $(u_n)$ and $u_0$ be in E and $u_n\rightharpoonup u_0$ weakly in $E$. Then,
\begin{eqnarray}
&& u_{n}\rightharpoonup u_0 \quad \text{weakly in}\, L^{s}(\mathbb{R}^3),\hspace{.17cm} 2\leq s\leq 6\nonumber\\
 && u_{n}\rightarrow u_0 \quad \text{in}\, L^{s}_{loc}(\mathbb{R}^3),\hspace{.17cm} 2\leq s < 6\label{loc}.
\end{eqnarray}

From Remark \ref{remark-phi-u}, $(\phi_{u_n})$ is bounded in $\mathcal{D}^{1,2}$.
So, there exists $\phi_0\in\mathcal{D}^{1,2}$ such that $\phi_{u_n}\rightharpoonup \phi_{0}$ in $\mathcal{D}^{1,2}$,  as a consequence,
\begin{eqnarray}
&&\phi_{u_n}\rightharpoonup \phi_{0} \quad \text{weakly in}\,\,
L^{6}(\mathbb{R}^3)\nonumber\\
&&\phi_{u_n}\rightarrow\phi_{0}\quad \text{in}\,\,
L^{s}_{loc}(\mathbb{R}^3), \hspace{.17cm} 1\leq s<6 \label{loc2}.
\end{eqnarray}

It remains to show that $\phi_{u_0}=\phi_{0}$. By Proposition (\ref{Propriedade-phi}), it suffices to show that $\phi_0$ satisfies
$\Delta\phi_0=(\omega+\phi_0)u_0^2$.

Let $\varphi\in C_0^{\infty}(\mathbb{R}^3)$ be a test function. Since $\Delta\phi_{u_n}=(\omega+\phi_{u_n})u_n^2$, we have
\begin{eqnarray*}-\int_{\mathbb{R}^3}\langle\nabla\phi_{u_n},
\nabla\varphi\rangle\,dx=\int_{\mathbb{R}^3}\omega\varphi u_n^2\,
dx+\int_{\mathbb{R}^3}\phi_{u_n}\varphi u_n^2.
\end{eqnarray*}

From (\ref{loc}), (\ref{loc2}) and the boundedness of ($\phi_{u_n})$ in $\mathcal{D}^{1,2}$, the following three sentences hold
\begin{eqnarray}\label{A}\begin{array}{rcl}
\displaystyle\int_{\mathbb{R}^3}\langle\nabla\phi_{u_n},\nabla\varphi\rangle\,dx
&\stackrel{ n\rightarrow\infty}{\longrightarrow} &
\displaystyle\int_{\mathbb{R}^3}\langle\nabla\phi_{0},\nabla\varphi\rangle\,dx\\
\displaystyle\int_{\mathbb{R}^3}\phi_{u_n} u_n^2\varphi\,dx
&\stackrel{ n\rightarrow\infty}{\longrightarrow}
\displaystyle & \displaystyle \int_{\mathbb{R}^3}\phi_{0} u_0^2\varphi\,dx\\
\displaystyle\int_{\mathbb{R}^3} u_n^2\varphi\,dx &\stackrel{
n\rightarrow\infty}{\longrightarrow}
&\displaystyle\int_{\mathbb{R}^3}u_0^2\varphi\,dx
\end{array}
\end{eqnarray}
proving that $\phi_{u_0}=\phi_{0}$.

As regards to the second part of the lemma, consider $v\in C_0^{\infty}(\mathbb{R}^3)$ a test function and observe that,
by the boundedness of ($\phi_{u_n})$, (\ref{loc}) and (\ref{loc2}),
\begin{eqnarray*}
\int_{\mathbb{R}^3}(\phi_{u_n}u_n-\phi_{0}u_0)v\, dx &=&
\int_{\mathbb{R}^3}\phi_{u_n}(u_n-u_0)v\, dx +
\int_{\mathbb{R}^3}u_0(\phi_{u_n}-\phi_0)v\, dx\\
&\leq & C\|\phi_{u_n}\|_{\mathcal{D}^{1,2}}\Big(
\int_{\mathbb{R}^3}|u_n-u_0|^{\frac{6}{5}}
|v|^{\frac{6}{5}}\,dx
\Big)^{\frac{5}{6}}+\int_{\mathbb{R}^3}(\phi_{u_n}-\phi_0)u_0v\,dx\\
&=& o_n(1).
\end{eqnarray*}
and
\begin{eqnarray*}
\int_{\mathbb{R}^3}(\phi_{u_n}^2 u_n-\phi_{0}^2 u_0)v\, dx &=&
\int_{\mathbb{R}^3}\phi_{u_n}^2(u_n-u_0)v\, dx +
\int_{\mathbb{R}^3}u_0(\phi_{u_n}^2-\phi_0^2)v\, dx\\
&\leq & C \|\phi_{u_n}\|_{\mathcal{D}^{1,2}}\Big(
\int_{\mathbb{R}^3}|u_n-u_0|^{\frac{3}{2}}|v|^{\frac{3}{2}}\,dx
\Big)^{\frac{2}{3}}+\int_{\mathbb{R}^3}(\phi_{u_n}^2-\phi_0^2)u_0 v\,dx\\
&=& o_n(1).
\end{eqnarray*}

Therefore,
\begin{eqnarray*}\int_{\mathbb{R}^3}(2\omega+\phi_{u_n})\phi_{u_n}u_nv\,dx-\int_{\mathbb{R}^3}(2\omega+\phi_{0})\phi_{0}u_0v\,dx=o_n(1),
\end{eqnarray*}
for all $ v\in C_0^\infty(\mathbb{R}^3)$. As a consequence, $\langle I'(u_n),v\rangle\longrightarrow\langle I'(u_0),v\rangle $ as $n\rightarrow\infty$.
\end{proof}

Consider
\begin{eqnarray}\label{alpha}
\alpha=\inf_{u\in\mathcal{N}}I(u).
\end{eqnarray}

We are going to prove that there exists $u_0\in\mathcal{N}$ with $I(u_0)=\alpha$, that is, $(u_0,\phi_{u_0})$ is a ground state solution of
(\ref{kgm}) system.

Assume $u_n\in\mathcal{N}$ such that $I(u_n)\rightarrow\alpha$, as $n\rightarrow\infty$.

From Lemma \ref{lema de compacidade}, there exist $C>0$, $r>0$ and a sequence $(\xi_n)\subset \mathbb{R}^3$ (we may assume without loss of generality that
$(\xi_n)\subset \mathbb{Z}^3$) such that
\begin{eqnarray*}\int_{B_{r+1}(\xi_n)}u_n^2\,dx \geq C>0.
\end{eqnarray*}

Define $v_n(x):=u_n(x-\xi_n)$. Since $V$ is 1-periodic and $\phi_{u_n}(x-\xi_n)=\phi_{v_n}(x)$, then
\begin{eqnarray*}
\|v_n\|=\|u_n\|, \quad  I(v_n)=I(u_n) \quad \mbox{and}\quad
I(v_n)\rightarrow\alpha,\,\, \mbox{as}\,\, n\rightarrow\infty.
\end{eqnarray*}

Moreover, from the boundedness of $(u_n)$ in $E$, $(v_n)$ is also bounded, from which we conclude that
\begin{eqnarray}\label{este}
\begin{array}{ll}
v_{n}\rightharpoonup v_0 & \mbox{weakly in}\,\, E,\\
v_{n}\rightarrow v_0, & \mbox{in}\,\, L_{loc}^{s}(\mathbb{R}^3), \,\, 1\leq s<6,\\
v_n\rightarrow v_0, &  \mbox{a.e. in} \,\,\mathbb{R}^3.
\end{array}
\end{eqnarray}

Now, in view of Lemma \ref{azz+pomp}, $\phi_{u_n}\rightharpoonup \phi_{0}$ in $\mathcal{D}^{1,2}$, then
\begin{eqnarray}\label{aqui2}\begin{array}{ll}
\phi_{u_n}\rightarrow \phi_0, & \mbox{in}\,\, L_{loc}^{s}(\mathbb{R}^3), \,\, 1\leq s<6,\\
\phi_{u_n}\rightarrow \phi_0, &  \mbox{a.e. in} \,\,\mathbb{R}^3.
\end{array}
\end{eqnarray}

Without loss of generality, we can assume that $(v_n)$ is a Palais-Smale sequence for the functional $I|_{\mathcal{N}}$, in particular,
\begin{eqnarray}\label{aqui*}
\begin{array}{ll}
I(v_n)\rightarrow\alpha, & \mbox{as}\,\, n\rightarrow \infty\\
(I|_{\mathcal{N}})'(v_n)\rightarrow 0, & \mbox{as}\,\,n\rightarrow\infty
\end{array}
\end{eqnarray}
then, for suitable Lagrange multipliers $\lambda_n$ we get
\begin{eqnarray*}
o_n(1)=\langle(I|_{\mathcal{N}})'(v_n),v_n\rangle= \langle I'(v_n),v_n\rangle+\lambda_n\langle G'(v_n),v_n\rangle=\lambda_n \langle G'(v_n),v_n\rangle.
\end{eqnarray*}

From Lemma \ref{nNehari}, we deduce that $\lambda_n=o_n(1)$ and by (\ref{aqui*}),
\begin{eqnarray*}
I'(v_n)\rightarrow 0, \,\,\mbox{as}\,\, n\rightarrow\infty.
\end{eqnarray*}

Using Lemma \ref{azz+pomp} and the last statement, we get $I'(v_0)=0$, where $v_0\neq 0$. Now we have to prove that $I(v_0)=\alpha$. But since
$I(v_n)\rightarrow\alpha$, it suffices to show that $I(v_n)\rightarrow I(v_0)$.

Since $v_n\in\mathcal{N}$, we have
\begin{eqnarray*}
I(v_n)=\frac{q-2}{2q}\|v_n\|+\frac{4-q}{2q} \int_{\mathbb{R}^{3}}\omega\phi_{u_n} v_n^2\,dx
+\frac{1}{q}\int_{\mathbb{R}^{3}}\phi_{u_n}^2 v_n^2\,dx +\frac{6-q}{6q}\int_{\mathbb{R}^{3}} v_n^{6}\,dx.
\end{eqnarray*}

Similarly as it was done by Azzollini and Pomponio in \cite{Azzollini-Pomponio-KGM}, we have to consider two cases: either $2<q<4$ or $4\leq q<6$. If
$4\leq q<6$, then by the weak lower semicontinuity of the $E$-norm, (\ref{este}), (\ref{aqui2}) and Fatou's Lemma we deduce that $I(v_0)\leq\alpha$.

On the other hand if $2<q<4$, using condition (V2) we get
\begin{eqnarray*}
\frac{q-2}{2q}\int_{\mathbb{R}^{3}}(|\nabla v_n|^2+ V(x)v_n^2)\,dx+\frac{4-q}{2q}\int_{\mathbb{R}^{3}}\omega\phi_{u_n} v_n^2 \,dx\geq 0
\end{eqnarray*}
and arguing as before, we again conclude that $I(v_0)\leq\alpha$.

But since $\alpha=\inf_{v\in\mathcal{N}}I(v)$, then $I(v_0)=\alpha$. Consequently, $(v_0,\phi_0)$ is a ground state solution for system (\ref{kgm}).

Using bootstrap arguments and the maximum principle, we can conclude that the solution $v_0$ is positive.

\bigskip

\paragraph{\textbf{Acknowledgements}} The authors would like to thank to Professor S. H. M. Soares for helpful comments.




\bibliographystyle{elsarticle-num}


\end{document}